%% file: boundarysplit.tex
\documentclass[final,leqno,onefignum,onetabnum]{siamltex1213}
\usepackage{amsmath}
\usepackage{amssymb}
\usepackage{graphicx}
\usepackage{algorithm}
\newcommand{\eps}{\varepsilon}
\newcommand{\wt}{\widetilde}

\newtheorem{example}[theorem]{\protect\examplename}
\providecommand{\examplename}{Example}

\providecommand{\remarkname}{Remark}
\newtheorem{assumption}[theorem]{\protect\assumptionname}
\providecommand{\assumptionname}{Assumption}

\title{Overcoming order reduction in diffusion-reaction splitting. Part 1: Dirichlet boundary conditions\thanks{This work is  supported by the Austrian Science Fund (FWF) -- project id: P25346.}}

\newcommand*\samethanks[1][\value{footnote}]{\footnotemark[#1]}
\author{
Lukas Einkemmer\thanks{Department of Mathematics, University of Innsbruck, Technikerstra\ss e 13, Innsbruck, Austria ({\tt lukas.einkemmer@uibk.ac.at}, {\tt alexander.ostermann@uibk.ac.at}).
} \and Alexander Ostermann\samethanks
}

\begin{document}

\maketitle

\slugger{sisc}{xxxx}{xx}{x}{x--x}

\begin{abstract}
For diffusion-reaction equations employing a splitting procedure is attractive as it reduces the computational demand and facilitates a parallel implementation. Moreover, it opens up the possibility to construct second-order integrators that preserve positivity. However, for boundary conditions that are neither periodic nor of homogeneous Dirichlet type order reduction limits its usefulness. In the situation described the Strang splitting procedure is not more accurate than Lie splitting. In this paper, we propose a splitting procedure that, while retaining all the favorable properties of the original method, does not suffer from order reduction. We demonstrate our results by conducting numerical simulations in one and two space dimensions with inhomogeneous and time dependent Dirichlet boundary conditions. In addition, a mathematical rigorous convergence analysis is conducted that confirms the results observed in the numerical simulations.
\end{abstract}



\pagestyle{myheadings}
\thispagestyle{plain}
\markboth{L.~EINKEMMER, A.~OSTERMANN}{OVERCOMING ORDER REDUCTION IN SPLITTING METHODS}

\section{Introduction}\label{sec:intro}

Splitting methods are considered a promising approach for the numerical
solution of diffusion-reaction problems (see, for example, \cite{spee1998}, \cite{gerisch2002},
or \cite{hundsdorfer2003numerical}). Such methods allow for a separate
treatment of the (linear) diffusion and the nonlinear, but local,
reaction.

The linear constant coefficient diffusion problem on a tensor product domain can be solved efficiently
by fast Fourier methods. On more complicated geometries or for space dependent coefficients an implicit
time marching scheme is necessary (due to the stringent stability
requirement for explicit schemes). The application of such a scheme
(for example, the Crank--Nicolson method) yields an elliptic system
of linear equations. Such systems can be efficiently solved by a collection
of techniques referred to as fast Poisson solvers such as multigrid
methods (see, for example, \cite{hackbusch1985}) or potential methods
(see, for example, \cite{mckenney1995fast}).

The reaction problem is local and thus all the degrees of freedom
decouple. This greatly aids the parallelization of the algorithm and
allows an efficient implementation even if the reaction is stiff (although
in this case order reduction due to the stiffness of the problem is
possible; see, for example, \cite{verwer1998}). We will not consider this here.

It is clear from the discussion above that the splitting approach
is advantageous from an implementation standpoint as one essentially
substitutes a large nonlinear system of equations by a linear system
that is treated with a fast Poisson
solver (with complexity $\mathcal{O}(n\log n)$, where $n$ is the
number of degrees of freedom) and a set of ordinary differential equations
(with complexity $\mathcal{O}(n)$). Also splitting methods preserve
positivity if the corresponding solvers of the partial flows have
this property; see~\cite{hansen2012}.

In the case of periodic or homogeneous Dirichlet boundary conditions
the well known Lie and Strang splittings are of order one and two,
respectively. Furthermore, splitting methods with complex coefficients
can be constructed that achieve arbitrary high order (at the cost
of using complex quantities in the intermediate steps; see, for example, \cite{Castella2009, hansen2009}). However, for
more general Dirichlet boundary conditions order reduction for the
Strang splitting to order one in case of the infinity norm\footnote{If the error
is measured in a discrete $L^p$ norm fractional orders between $1$ and $2$ are
observed. See section \ref{sec:Numerical-results} for more details.}
is observed. Thus the Strang splitting
scheme is not more accurate than the Lie splitting scheme (see, for
example, \cite{hundsdorfer2003numerical}). Similar order reductions
for an advection-reaction problem have been observed in \cite{hundsdorfer1995}.
In the before mentioned paper a remedy has been proposed. However,
it is not clear how to extend this approach beyond the toy problem
considered there.

In this paper we consider the diffusion-reaction initial-boundary
value problem given by
\begin{equation}\label{eq:di-rea}
\begin{aligned}
\partial_{t}u & =Du+f(u),\\
u\vert_{\partial\Omega} & =b,\\
u(0) & =u_{0},
\end{aligned}
\end{equation}
where $D$ is an elliptic differential operator (for example, the Laplacian) and
$f\colon\mathbb{R}\to\mathbb{R}$ is the reaction term. We consider the domain $\Omega\subset\mathbb{R}^{d}$,
initial value $u_{0}$, and Dirichlet boundary conditions given by
$b\colon[0,T]\times\partial\Omega\to\mathbb{R}$. Note that, in general,
$b$ is allowed to depend on time.
Equation~\eqref{eq:di-rea} could equally represent the spatial discretization
of a diffusion-reaction problem. In that case, $D$ is a matrix with large norm and
the application of $f$ is understood componentwise. Usually the boundary condition
is included in $D$ in the (space) discretized equations. However,
for the discussion that follows we believe it is more useful to keep
the boundary condition separate. This also enables us to consider
the semi-discrete system (i.e., where time is discretized but space is left
continuous). Thus, in both the continuous and discrete case $D$ only models the
differentiation in the interior of the domain and is thus a non-invertible
linear operator or matrix.

In section \ref{sec:Theory} we will propose an alternate splitting
which does not suffer from the order reduction and thus significantly
increases the efficiency of the numerical integrator under consideration.
Numerical experiments will be performed for a variety of configurations
in both a single and two space dimensions (see sections \ref{sec:Numerical-results} and \ref{sec:Numerical-results-(2D)}).
For the one-dimensional examples we employ a finite difference approximation,
while in the two-dimensional case we will use a finite element space
discretization. In section~\ref{sec:convergence} we provide a rigorous
convergence analysis that confirms and explains the behavior observed in the numerical simulations.

\section{Description of the numerical method\label{sec:Theory}}

A number of numerical experiments has been conducted in the literature
that show order reduction for splitting methods applied to evolution problems
(see the results gathered in \cite{hundsdorfer2003numerical}). The
observation is made that for advection-reaction and diffusion-reaction
equations the order reduction is usually not present for homogeneous
Dirichlet boundary conditions. However, even for very simple non-homogeneous
boundary conditions Strang splitting is only of order one (let us
also refer to the numerical simulations conducted in the next section).
In the splitting procedure no boundary condition can be imposed for
the reaction term, while the boundary condition for the diffusion
term is the same as that for the original problem (Dirichlet boundary
conditions in our case).

In \cite{hansen2012} a convergence proof for diffusion-reaction problems
in an abstract setting was conducted. Among the assumptions of the
proof that the Strang splitting scheme is of order two is the requirement
that the evolution of the reaction partial flow leaves the domain
of $L^{2}$ invariant, where $L$ denotes a second-order strongly elliptic
differential operator (e.g., the Laplacian) endowed with the appropriate
boundary conditions. The required differentiability is usually no problem.
However, the domain
of $L^{2}$, denoted by $\mathcal{D}(L^{2})$, depends crucially on
the boundary condition as well. In most cases physically it is required
that $f(0)=0$. That is, if the concentration of a given substance
is zero, the reaction can not change this state. Then, for homogeneous
Dirichlet boundary conditions $\mathcal{D}(L^{2})$ is left invariant
by the reaction partial flow. For other boundary conditions, however,
this is not the case and order reduction is encountered.

Therefore, we propose to rewrite the problem in such a way that homogeneous
boundary conditions can be imposed. To that end, let us introduce
a function $z$ that is determined by the following elliptic problem
\begin{align*}
Dz & =0,\\
z\vert_{\partial\Omega} & =b.
\end{align*}
That is, $z$ is the harmonic (in case of the Laplacian) continuation
of the boundary data $b$. Then, let us define $\wt{u}=u-z$ which satisfies
\begin{equation}\label{eq:u-tilde}
\begin{split}
\partial_{t}\wt{u} & =D\wt{u}+f(\wt{u}+z)-\partial_{t}z,\\
\wt{u}\vert_{\partial\Omega} & =0,\\
\wt{u}(0) &= u_0-z_0.
\end{split}
\end{equation}
A similar approach has been considered in \cite{hundsdorfer1995}
for an advection-reaction equation. There the variation-of-constants
formula is now applied which yields an expansion which can be compared
to the exact solution. This then suggests a modification of the splitting
procedure. However, in that case we have to integrate backward in
time and it is not clear if this approach can be extended to a diffusion-reaction
equation.

\begin{algorithm}[t]
\caption{\quad Modified Lie splitting for~\eqref{eq:di-rea}\label{alg:lie}}
\begin{enumerate}
\item Solve $Dz_{0}=0$ using the boundary condition $z_{0}\vert_{\partial\Omega}=b(0)$.
\item Compute the initial value $\wt{w}(0)=u_{0}-z_{0}$.
\item Compute the solution of (\ref{eq:splitstep-B}) to obtain $\wt{w}(\tau)$.
\item Compute the solution of (\ref{eq:splitstep-A}) with initial value
$\wt{w}(\tau)$ using homogeneous Dirichlet boundary conditions
to obtain $\wt{v}(\tau)$.
\item Solve $Dz_{1}=0$ using the boundary condition $z_{1}\vert_{\partial\Omega}=b(\tau)$.
\item Set $u_1=\wt{v}(\tau)+z_{1}$.
\end{enumerate}
\end{algorithm}

We propose to apply a splitting directly to (\ref{eq:u-tilde}). However,
we still need a compatibility condition for the nonlinearity; that
is, we want to split the nonlinearity $f(\wt{u}+z)$ into a term $g(t,\wt{u})$ such
that $g(t,0)=0$ and a second term that does not depend on $\wt{u}$. These
requirements lead to the obvious choice of the two partial flows
given by
\begin{equation}\label{eq:splitstep-A}
\begin{aligned}
\partial_{t}\wt{v}&=D\wt{v}+f(z)-\partial_{t}z,\\
\wt{v}\vert_{\partial\Omega} & =0\\
\end{aligned}
\end{equation}
and
\begin{equation}\label{eq:splitstep-B}
\partial_{t}\wt{w}=f(\wt{w}+z)-f(z),
\end{equation}
respectively. The modified nonlinearity is now given by $g(t,u)=f(u+z(t))-f(z(t))$. Its
explicit time dependence is a consequence of the (potential) time dependence of $z$.
The new nonlinearity satisfies $g(t,0)=0$ as required.

One time step of size $\tau$ from $t=0$ to $t=\tau$ of the Lie splitting scheme with
initial value $u_{0}$ proceeds as shown in Algorithm~\ref{alg:lie}. Here, we started
with the nonlinear flow, followed by the linear one. The corresponding adaption for
Strang splitting is obvious. The method that starts with a half step of the linear
flow is given in Algorithm~\ref{alg:strang}.

\begin{algorithm}[t]
\caption{\quad Modified Strang splitting for~\eqref{eq:di-rea}\label{alg:strang}}
\begin{enumerate}
\item Solve $Dz_{0}=0$ using the boundary condition $z_{0}\vert_{\partial\Omega}=b(0)$.
\item Compute the initial value $\wt{v}(0)=u_{0}-z_{0}$.
\item Compute the solution of (\ref{eq:splitstep-A}) using homogeneous Dirichlet boundary conditions
to obtain $\wt{v}(\frac{\tau}2)$.
\item Compute the solution of (\ref{eq:splitstep-B}) with initial value $\wt{w}(0)= \wt{v}(\frac{\tau}2)$
to obtain $\wt{w}(\tau)$.
\item Compute the solution of (\ref{eq:splitstep-A}) with initial value
$\wt{w}(\tau)$ using homogeneous Dirichlet boundary conditions
to obtain $\wt{v}(\frac{\tau}{2})$.
\item Solve $Dz_{1}=0$ using the boundary condition $z_{1}\vert_{\partial\Omega}=b(\tau)$.
\item Set $u_1=\wt{v}(\frac{\tau}{2})+z_{1}$.
\end{enumerate}
\end{algorithm}

The crucial point here is that the modifications added to the discretized
Laplacian in equation (\ref{eq:splitstep-A}) do not negatively
impact our ability to efficiently compute a numerical approximation
as only a position (and possibly time) dependent source
term is added. This poses no additional difficulty for applying fast
Fourier methods or most fast Poisson solvers. More generally, we can employ numerical
methods referred to as exponential integrators (see, for example, \cite{hochbruck2010})
to approximate the solution of
\begin{equation}\label{eq:spliA-0}
\partial_{t}\wt{v}=L\wt{v}+f(z)-\partial_{t}z,
\end{equation}
where $L$ denotes the operator $D$ equipped with homogeneous Dirichlet boundary conditions.
If the boundary conditions in \eqref{eq:di-rea} are time invariant
(the simplification we will consider in the following example) the exponential Euler method
\[
\wt{v}(t) = \mathrm{e}^{t L}\wt{v}(0) + t \varphi_1 (t L) f(z),
\]
where $\varphi_1$ is an entire function of $L$, is exact. For time dependent boundary
conditions a second-order exponential integrator can be employed (for more details see \cite{hochbruck2010}).
A disadvantage of this approach is that due to the requirement of evaluating the $\varphi_1$ function a true black-box solver for \eqref{eq:spliA-0} can not be used. Furthermore, preconditioning is difficult in this formulation. In such a case we can employ a class of methods referred to as IMEX (IMplicit EXplicit). In this case the operator $L$ in \eqref{eq:spliA-0} is treated implicitly (ideally with a good preconditioner) while the additional non-stiff term is integrated explicitly. For a more detailed discussion see \cite{garcia2014}.

\begin{example}[Time independent boundary conditions] An important simplification
constitutes the case where $b$ is independent of time. In this case $z$ is
also independent of time and can be precomputed. Moreover, by setting
$\wt{v} = v-z$, $\wt{w} = w-z$, we can perform the splitting more directly. Instead of \eqref{eq:splitstep-A}, \eqref{eq:splitstep-B} we simply consider
\begin{align*}
\partial_{t}v & =Dv+f(z)\\
v\vert_{\Omega} & =b
\end{align*}
and
\[
\partial_{t}w=f(w)-f(z).
\]
\end{example}

\section{Convergence analysis}\label{sec:convergence}

In light of the method described in the previous section, let us consider the following abstract evolution equation
\begin{equation}\label{eq:abstract-evolution}
\begin{aligned}
	\partial_t u &= Au + g(t,u) + k(t), \\
	u(0) &= u_0.
\end{aligned}
\end{equation}
This is in fact problem \eqref{eq:u-tilde} with the boundary conditions included in the domain of the operator $A$. For example, in the case of a strongly elliptic second-order differential operator $D$ on $L^2(\Omega)$, it holds that $\mathcal D(A) = H^2(\Omega)\cap H^1_0(\Omega)$ and $A\psi=D\psi$ for all test functions in $\Omega$.

In this situation, $A$ is the infinitesimal generator of an analytic semigroup $\mathrm{e}^{tA}$ and there exists a constant $\omega\ge 0$ such that the fractional powers $(\omega I - A)^\alpha$ are well defined for $\alpha\in \mathbb R$; see, \cite[Chap.~1.4]{henry81}. By a simple rescaling argument one can always take $\omega = 0$. This will be done henceforth.

Let us now proceed by splitting equation \eqref{eq:abstract-evolution} into the two partial flows given by
\[
\partial_t v(t)=Av(t)+k(t)
\]
and
\[
\partial_t w(t) = g(t,w(t)),
\]
respectively. Depending on the choice of $g$ and $k$ this represents the classical splitting ($g=f$ and $k=0$) or the modified splitting ($g=f-k$). In the latter case $k$ is chosen such that that the compatibility condition $g(t,0)=0$ is satisfied.

\subsection{Lie and modified Lie splitting}
In the above setting the Lie splitting operator $\mathcal L_\tau$ is given by\footnote{One could also reverse the order and consider the splitting
$$
\mathcal L_\tau z = \varphi_{\tau}^{g}\bigl(\varphi^{A,k}_\tau(z) \bigr)
$$
instead. As its analysis is very similar to that of \eqref{eq:L-Split}, we do not consider it here.}
\begin{equation}\label{eq:L-Split}
\mathcal L_\tau z = \varphi^{A,k}_\tau\bigl( \varphi_{\tau}^{g}(z) \bigr),
\end{equation}
where $\varphi_{\tau}^g(z)$ denotes $w(t_n+\tau)$ with initial value $w(t_n)=z$ and $\varphi_{\tau}^{A,k}(z)$ denotes $v(t_n+\tau)$ with initial value $v(t_n)=z$.

In order to analyze the splitting scheme we first consider its local error. Thus, let us express $v$ as
\[
v(t_n+\tau) = \varphi_\tau^{A,k}(z) = \mathrm{e}^{\tau A} z + \int_0^{\tau} \mathrm{e}^{(\tau-s)A} k(t_n+s)\,\mathrm{d}s
\]
and $w$ as
\[
w(t_n+\tau) = z + \tau g(t_n,z) + \int_0^\tau (\tau-s)w^{\prime \prime}(t_n+s)\,\mathrm{d}s.
\]
For the Lie splitting scheme this gives
\begin{equation}\label{eq:lie-entwickelt}
\begin{aligned}
\mathcal L_{\tau}z = \mathrm{e}^{\tau A}z+\tau \mathrm{e}^{\tau A} g(t_n,z) &+ \int_0^{\tau} \mathrm{e}^{(\tau-s)A} k(t_n+s)\,\mathrm{d}s\\
& + \int_0^\tau \mathrm{e}^{\tau A}(\tau-s)w^{\prime \prime}(t_n+s)\,\mathrm{d}s.
\end{aligned}
\end{equation}
Now, let us expand the exact solution of equation \eqref{eq:abstract-evolution} with initial value $u(t_n)=z$
\begin{equation}\label{eq:expansion-exact-1}
\begin{aligned}
u(t_n+\tau) = \mathrm{e}^{\tau A} z &+ \int_0^{\tau} \mathrm{e}^{(\tau-s)A} k(t_n+s)\,\mathrm{d}s\\
&+ \int_0^\tau \mathrm{e}^{(\tau-s)A} g(t_n+s,u(t_n+s))\,\mathrm{d}s.
\end{aligned}
\end{equation}
Combining these results we get for the local error
\begin{subequations}\label{eq:local-error}
\begin{equation}
\mathcal L_{\tau}u(t_n) - u(t_n+\tau)= \int_0^\tau \mathrm{e}^{\tau A}(\tau-s) w^{\prime \prime}(t_n+s)\,\mathrm{d}s -
\int_0^\tau \int_0^s \ell_n^{\prime}(\xi) \,\mathrm{d}\xi \mathrm{d}s,
\end{equation}
where
\begin{equation}\label{eq:le-ln}
\ell_n(s) = \mathrm{e}^{(\tau-s) A}g(t_n+s,u(t_n+s)).
\end{equation}
\end{subequations}
What we observe here is that to bound $\ell_n^{\prime}$ we need to bound $A g(t,u(t))$. This can be achieved for the modified splitting for sufficiently smooth $g$ as the compatibility condition at the boundary is satisfied. Thus, we conclude that the modified splitting has a consistency error proportional to $\tau^2$. For the classical splitting, however, $A g(t,u(t))$ can not be bounded as $g(t,0)\ne 0$ in general. Thus, the classical Lie splitting scheme has a consistency error proportional to $\tau$ only. However, in the numerical simulations conducted in the next section we will observe that also the classical splitting is convergent of order one. We will now explain this behavior as a consequence of the parabolic smoothing property.

Henceforth, we will employ the following assumption on the data of~\eqref{eq:di-rea}.

\begin{assumption} \label{ass1}
Let $D$ be a strongly elliptic differential operator with smooth coefficients, $f$ continuously differentiable, $b$ continuous in $t$, and assume that $u_0$ is spatially smooth and satisfies the boundary conditions.
\end{assumption}

Under these assumptions $A$ generates an analytic semigroup, $g$ is continuously differentiable and $k$ is continuous. Moreover, as a consequence of \cite[Thm.~3.5.2]{henry81}, the solution $u$ of \eqref{eq:abstract-evolution} is continuously differentiable.

\begin{theorem}[Convergence of the classical Lie splitting] \label{thm:lie-classic-o1}
Under Assumption~\ref{ass1}, the classical Lie splitting is convergent of order $\tau \left\vert \log \tau \right\vert$, i.e., the global error satisfies the bound
$$
\|u_n-u(t_n)\| \le C\tau (1+\left\vert \log \tau \right\vert), \qquad 0\le n\tau \le T,
$$
where the constant $C$ depends on $T$ but is independent of $\tau$ and $n$.
\end{theorem}

\begin{proof}
Note that the classical Lie splitting corresponds to the choice of $k=0$ in equation \eqref{eq:abstract-evolution}. Thus, the local error is given by equation \eqref{eq:local-error}. Let us denote the global error by $e_n=u_n-u(t_n)$. Then
\[
e_{n+1} = \mathcal L_\tau u_n - \mathcal L_\tau u(t_n) + d_{n+1},
\]
where $d_{n+1}$ denotes the local error. Now, we have
\begin{align*}
\mathcal L_{\tau}u_n - \mathcal L_{\tau}u(t_n) &= \mathrm{e}^{\tau A} \bigl( \varphi_\tau^g(u_n) - \varphi_\tau^g(u(t_n)) \bigr) \\
&= \mathrm{e}^{\tau A}e_n + \tau E(u_n,u(t_n)),
\end{align*}
where due to the Lipschitz continuity of $g$ it holds that $\Vert E(u_n,u(t_n)) \Vert \leq C \Vert e_n \Vert$. Inserting this into the recurrence relation for the global error gives
\[
e_{n+1} = \mathrm{e}^{\tau A}e_n + d_{n+1} + \tau E(u_n,u(t_n)).
\]
The crucial point here is that we now solve only for the linear part (as in this case we know that the parabolic smoothing property holds true). This gives
\[
e_{n} =  \mathrm{e}^{n\tau A}e_0 + \sum_{k=1}^n \mathrm{e}^{(n-k)\tau A}d_k + \tau \sum_{k=0}^{n-1} \mathrm{e}^{(n-k-1)\tau A}E(u_k,u(t_k)).
\]
Using the parabolic smoothing property for the linear evolution, i.e.~using that
\[
\Vert \mathrm{e}^{t A}(-A)^\alpha \Vert \leq Ct^{-\alpha},\qquad\alpha\ge 0
\]
for all $t\in(0,T]$, we get
\[
\Vert e_n \Vert \leq C\|e_0\| + C \tau^2 \sum_{k=1}^{n-1} \frac{1}{k \tau} + C\tau + C\tau \sum_{k=0}^{n-1} \Vert e_k \Vert,
\]
where the second term can by estimated by $C \tau (1+\left\vert \log \tau \right\vert)$ which together with Gronwall's inequality (using $\|e_0\|=0$) gives the desired bound.
\end{proof}

Note that in the setting of Theorem \ref{thm:lie-classic-o1} the classical Lie splitting is consistent of order zero (that is, the local error is proportional to $\tau$). However, due to the parabolic smoothing property we can employ the expansion up to $\tau^2$ and bound the remainder. The same proof can be conducted in order to show that the modified Lie splitting is convergent of order one. However, in this case a more direct proof is possible as the method is consistent of order one. This gives the following result.
\begin{theorem}
Under Assumption~\ref{ass1} 
the modified Lie splitting is first-order convergent.
\end{theorem}

\subsection{Strang and modified Strang splitting}
The above analysis for the Lie splitting also explains the behavior of the Strang splitting. For the classical Strang splitting the local error is not improved and thus we still only obtain order one (in the maximum norm). On the other hand the modified Strang splitting is consistent of order one but, due to the parabolic smoothing property, is convergent of order two. This behavior can be easily observed in numerical tests (see Table~\ref{fig:consistency}). To show this analytically is the purpose of this section.

The Strang splitting scheme for \eqref{eq:abstract-evolution} is defined by\footnote{The version of the Strang splitting with the reversed order of the partial flows can be analyzed by proceeding in a similar fashion.}
\[
\mathcal S_\tau z = \varphi_{\tau/2}^{A,k}\Big( \varphi_\tau^g\Big(\varphi_{\tau/2}^{A,k}(z)\Big) \Big).
\]
Thus, we have to solve in succession the following abstract initial value problems
\begin{alignat*}{2}
	\partial_t \overline{v}(t) &= A\overline{v}(t)+k(t),& \overline{v}(t_n)&=z \\
	\partial_t w(t) &= g(t,w),&\qquad w(t_n)&= \overline{z} = \overline{v}(t_n+\tfrac{\tau}2)\\
	\partial_t \overline{\overline{v}}(t) &= A\overline{\overline{v}}(t) + k(t), &\qquad\quad \overline{\overline{v}}(t_n+\tfrac{\tau}{2})&=w(t_n+\tau)
\end{alignat*}
which can be expanded as
\[ \overline{v}(t_n+\tfrac{\tau}{2}) = \mathrm{e}^{\frac{\tau}{2} A}z
+ \int_0^{\frac{\tau}{2}} \mathrm{e}^{(\frac{\tau}{2}-s)A}k(t_n+s)\,\mathrm{d}s
\]
and
\[ w(t_n+\tau) = \overline{z} + \tau g(t_n,\overline{z}) + \frac{\tau^2}{2}w^{\prime \prime}(t_n)
+ \frac{1}{2} \int_0^\tau (\tau-s)^2 w^{(3)}(t_n+s)\,\mathrm{d}s
\]
and
\[ \overline{\overline{v}}(t_n+\tau) = \mathrm{e}^{\frac{\tau}{2} A}w(t_n+\tau)
+ \int_0^{\frac{\tau}{2}} \mathrm{e}^{(\frac{\tau}{2}-s)A}k(t_n+\tfrac{\tau}{2}+s)\,\mathrm{d}s,
\]
respectively. Combining these expressions we get
\begin{subequations}
\begin{equation} \label{eq:strang-expansion}
\begin{aligned}
\mathcal S_{\tau}z = \mathrm{e}^{\tau A}z &+ \int_0^\tau \mathrm{e}^{(\tau-s)A}k(t_n+s)\,\mathrm{d}s +\tau \mathrm{e}^{\frac{\tau}{2}A}g(t_n,X)\\[1mm]
&+ \tfrac{\tau^2}{2}\mathrm{e}^{\frac{\tau}{2}A} \bigl( \partial_1 g(t_n,X) + \partial_2 g(t_n,X)g(t_n,X) \bigr) + \mathcal{O}(\tau^3)
\end{aligned}
\end{equation}
with
\begin{equation} \label{eq:strang-expansion-X}
X = \mathrm{e}^{\frac{\tau}{2}A}z + \int_0^{\frac{\tau}{2}} \mathrm{e}^{(\frac{\tau}{2}-s)A}k(t_n+s)\,\mathrm{d}s,
\end{equation}
\end{subequations}
where $\partial_1 g$ and $\partial_2 g$ denote the derivatives of $g$ with respect to the first and second argument, respectively. Note that the bounded remainder term, denoted by $\mathcal{O}(\tau^3)$, does not include any application of $A$. We will employ this notation in the remainder of this section.

Now, consider the expansion of the exact solution given in equation \eqref{eq:expansion-exact-1}. In case of the Lie splitting we simply used a Taylor series expansion at the left point of the interval under consideration. However, the third term in equation \eqref{eq:strang-expansion} suggests a symmetric approach. Therefore, we use the mid-point rule to obtain
\begin{equation}\label{eq:midpoint}
\begin{split}
\int_0^\tau \mathrm{e}^{(\tau-s)A} g\bigl(t_n+s,u(t_n+s)\bigr)\,\mathrm{d}s &= \tau \mathrm{e}^{\frac{\tau}{2}A} g\bigl(t_n+\tfrac{\tau}{2},u(t_n+\tfrac{\tau}{2})\bigr)  \\
&\qquad + \tfrac12 \int_0^\tau \! K(s,\tau)\,\ell_n^{\prime \prime}(s)\,\mathrm{d}s
\end{split}
\end{equation}
with the kernel $K(s,\tau)=s^2/2$ for $s<\tau/2$ and $K(s,\tau)=(\tau-s)^2/2$ for $s>\tau/2$, and $\ell_n$ as in~\eqref{eq:le-ln}. The remainder term will be discussed in some detail in the proof of
Theorem~\ref{thm:convergence-modified-strang}.

What remains to complete the consistency argument is to compare the first term on the right-hand side of equation \eqref{eq:midpoint} with the third and fourth term on the right-hand side of equation \eqref{eq:strang-expansion}. By using equation \eqref{eq:expansion-exact-1}, we obtain
\begin{align*}
g(t_n+\tfrac{\tau}{2},u(t_n+\tfrac{\tau}{2})) &- g(t_n,X) \\
&= \tfrac{\tau}{2}\partial_1 g(t_n,X) + \partial_2 g(t_n,X)\bigl(u(t_n+\tfrac{\tau}{2})-X\bigr) + \mathcal{O}(\tau^2) \\[1.5mm]
&= \tfrac{\tau}{2}\partial_1 g(t_n,X) + \tfrac{\tau}{2}\partial_2 g(t_n,X)g(t_n,u) \\
&\qquad +  \partial_2 g(t_n,X)\int_0^\tau\! \int_0^s \ell_n^{\prime}(\xi)\,\mathrm{d}\xi\mathrm{d}s + \mathcal{O}(\tau^2),
\end{align*}
which is the desired result.

We will employ the following assumption on the data of~\eqref{eq:di-rea}.
\begin{assumption} \label{ass2}
Let $D$ be a strongly elliptic differential operator with smooth coefficients, $f$ twice continuously differentiable, $b$ continuously differentiable, and let us assume that $u_0$ and $Du_0$ are spatially smooth and satisfy the boundary conditions.
\end{assumption}

We are now in the position to state the convergence result for Strang splitting.

\begin{theorem}[Convergence of the modified Strang splitting] \label{thm:convergence-modified-strang}
Under Assumption~\ref{ass2} the modified Strang splitting scheme is convergent of order $\tau^2 \left\vert \log \tau \right\vert$, i.e., the global error satisfies the bound
$$
\|u_n-u(t_n)\| \le C\tau^2 (1+\left\vert \log \tau \right\vert), \qquad 0\le n\tau \le T,
$$
where the constant $C$ depends on $T$ but is independent of $\tau$ and $n$.
\end{theorem}

\begin{proof}
Due to the fulfilled compatibility condition $g(t,0)=0$ we can bound a single application of $A$. This, however, is not sufficient as the remainder term in \eqref{eq:midpoint} includes an expression of the form
\[
\mathrm{e}^{(\tau-s)A}A^2g(t_n+s,u(t_n+s)).
\]
Thus, the modified Strang splitting is only consistent of order one (i.e., the local error is proportional to $\tau^2$). Now, similarly to the proof of Theorem \ref{thm:lie-classic-o1} we can use parabolic smoothing to bound the application of the remaining $A$. This shows that the modified Strang splitting is convergent of order two.
\end{proof}

In the following section (see Table~\ref{fig:l1l2}) we will present numerical simulations that show order reduction for the classical Strang splitting to approximately order $1.5$ for the discrete $L^1$ norm and to order $1.25$ for the discrete $L^2$ norm. Such an error behavior can be explained as follows: recall that we have to bound terms of the form
\[
I=\sum_{k=0}^{n-1} \mathrm{e}^{(n-k-1)\tau A} \int_0^\tau \! K(s,\tau)\,\mathrm{e}^{(\tau-s)A}A^2g(t_k+s,u(t_k+s))\,\mathrm{d}s.
\]
For this purpose, we have to estimate
$$
I_k = \int_0^\tau \! K(s,\tau)\,\mathrm{e}^{(\tau-s) A}(-A)^{p+\eps} (-A)^{1-p-\eps} g(t_k+s,u(t_k+s))\,\mathrm{d}s.
$$
Using the explicit form of the kernel, the parabolic smoothing property and the fact that a spatially smooth function lies in the domain of $(-A)^{1-p-\eps}$ with $p=\tfrac{1}{2}$ in $L^1$ and $p=\tfrac{3}{4}$ in $L^2$ for $\eps>0$ arbitrarily small (see \cite{fujiwara1967, grisvard1967}) we obtain that
$$
\|I_k\| \le C \tau^{3-p-\epsilon},\qquad \|AI_k\|\le C \tau^{2-p-\epsilon}.
$$
Taking all these bound together and using once more the parabolic smoothing property, we get
\begin{align*}
\|I\| &\le C \sum_{k=0}^{n-2} \|\mathrm{e}^{(n-k-1)\tau A}A \|\cdot \|I_k\| + \|AI_{n-1}\|\\
&\le C \tau^{2-p-\epsilon}.
\end{align*}
This argument proves the orders observed in Table~\ref{fig:l1l2}~(left) even without requiring the compatibility condition $g(t,0)=0$.

\begin{table}
\protect\caption{Problem \eqref{eq:di-rea} with $f(u)=u^2$, $u_{0}(x)=1+\sin^{2}\pi x$, and $b=1$ is discretized in
space using $500$ grid points. The error in the discrete infinity norm is computed after a single time step by
comparing the numerical solution to a reference solution which uses multiple smaller time steps. \label{fig:consistency}}
\begin{centering}
\input{table-consistency.tex}
\par\end{centering}
\end{table}

\section{Numerical results (1D)}\label{sec:Numerical-results}

In this section we will present a number of numerical results for
the diffusion-reaction problem \eqref{eq:di-rea} with
$$
f(u)=u^2
$$
on $\Omega=[0,1]$, where $D$ is the classical centered second-order
finite difference approximation of the Laplacian. Let us denote the
value of the solution at the left boundary by $b_{0}$ and at the right
boundary by $b_{1}$. In all the simulation we will refer to the classical
splitting approach by Lie and Strang, respectively, while we refer
to the schemes introduced in section~\ref{sec:Theory} by Lie (modified)
and Strang (modified), respectively.

\begin{table}[t]
\caption{Diffusion-reaction equation with $u_{0}(x)=1+\sin^{2}\pi x$, $500$ grid points,
and $b_{0}=b_{1}=1$. The error in the discrete infinity norm
is computed at $t=0.1$ by comparing the numerical solution to a reference solution computed
with the modified Strang splitting. \label{fig:Dirichlet=00003D1}}
\begin{centering}
\input{table-1d-1.tex}
\par
\end{centering}
\end{table}

\begin{example}[One-dimensional problem with $b_{0}=1$, $b_{1}=1$]
Even for this simple problem we can clearly observe reduction to order
one for the Strang splitting. The numerical results are given in Table
\ref{fig:Dirichlet=00003D1}. We observe that for the Lie splitting
the modified scheme results in a decrease in the error by about a
factor of $2$ compared to the classical Lie splitting. The modified
Strang splitting is a method of order two.
In Table~\ref{fig:l1l2} the error in the discrete $L^1$ and $L^2$ norm is shown for the same configuration. As expected, in the discrete $L^1$ norm we observe reduction to approximately $1.5$ (for the Strang splitting scheme), whereas in the discrete $L^2$ norm we observe order reduction to approximately order $1.25$.

\begin{table}
\caption{Diffusion-reaction equation with $u_{0}(x)=1+\sin^{2}\pi x$, $500$ grid points,
and $b_{0}=b_{1}=1$. The error in the discrete $L^1$ and discrete $L^2$ norm
is computed at $t=0.1$ by comparing the numerical solution to a reference solution computed
with the modified Strang splitting. \label{fig:l1l2}}
\begin{centering}
\input{table-1d-1-l1l2.tex}
\par\end{centering}
\end{table}

\end{example}

\begin{example}[One-dimensional problem with time dependent $b_{0}=b_{1}$]
We now consider a time dependent problem where both the left and the
right boundary are set to $b_0(t) = b_1(t) = 1+\sin 5t$.
The numerical results are given in Table~\ref{fig:Dirichlet-timedep1}.
They show order two for the modified Strang splitting and order one
for the classical Strang splitting.

\begin{table}
\protect\caption{Diffusion-reaction equation with $u_{0}(x)=1+\sin^{2}\pi x$, $500$ grid points,
and $b_{0}(t)=b_{1}(t)=1+\sin 5t$. The error in the discrete infinity
norm is computed at $t=0.1$ by comparing the numerical solution to a reference solution
computed with the modified Strang splitting. \label{fig:Dirichlet-timedep1}}
\begin{centering}
\input{table-1d-timedep1.tex}
\par\end{centering}
\end{table}

\end{example}

\begin{example}[One-dimensional problem with one constant and one
time dependent boundary condition] In this example we consider a
fixed left boundary condition $b_{0}=0.5$ and a time dependent right boundary
condition $b_{1}=1+\sin 20\pi t$.
The numerical results are shown in Table \ref{fig:Dirichlet-timedep2}.
This proves to be a more challenging numerical test. However, as before,
the observed results show order two for the modified Strang splitting
and only order one for the classical Strang splitting scheme.

\begin{table}
\protect\caption{Diffusion-reaction equation with $u_{0}(x)=\tfrac{1}{2}+\tfrac{1}{2}x$,
$500$ grid points, and $b_{0}(t)=0.5$, $b_{1}(t)=1+\sin 20\pi t$. The error in the discrete
infinity norm at $t=0.1$ is determined by comparing the numerical solution to a reference
solution with step size $\tau=5\cdot10^{-5}$ computed with the Strang
and modified Strang splitting, respectively. \label{fig:Dirichlet-timedep2}}
\begin{centering}
\input{table-1d-timedep2.tex}
\par\end{centering}
\end{table}

\end{example}

\section{Numerical results (2D)}\label{sec:Numerical-results-(2D)}

In this section we will present a number of numerical results for
the problem given in \eqref{eq:di-rea} with
\begin{equation}\label{eq:data2d}
f(u)= u^2, \qquad b=u_{0}\vert_{\partial\Omega}
\end{equation}
on $\Omega=[0,1]^{2}$, where $D$ is a finite element approximation
of order $2$ of the Laplacian (we use the libmesh finite element
library). Thus, we will limit ourselves here to the case of time independent
boundary conditions. In all simulations we will use the classical Runge--Kutta
method of order four to integrate the nonlinearity in time and the
Crank--Nicolson method to integrate the linear diffusion. In the latter
case we conduct $10$ substeps per splitting step. This allows us
to observe the error due to the splitting method only (and avoid any
interference from the second-order error of the Crank--Nicolson method).
The continuation $z$ is precomputed by a Poisson solver and is subsequently
used in each time step.

\begin{example}[Two-dimensional problem with $b=1$] This example is an
extension of the one-dimensional problem. We set the boundary condition
equal to $1$ everywhere and choose as the initial value
\[
u_{0}(x,y)=1+\sin^{2}(\pi x)\sin^{2}(\pi y).
\]
The numerical results are shown in Table \ref{fig:table-2d-1} and
confirm the order reduction in case of the classical Strang splitting
as well as that the modified Strang scheme proposed in this paper is of
second order. Let us also note that the modified Lie splitting is more accurate
by about a factor of $\,3.5$ as compared to the classical Lie splitting.

\begin{table}
\protect\caption{Diffusion-reaction equation with $u_{0}(x)=1+\sin^{2}(\pi x)\sin^{2}(\pi y)$
and $b=1$. For the Lie splitting scheme $10^{4}$ quadrilateral finite
elements are employed, while for the Strang splitting scheme $2.5\cdot10^{5}$
quadrilateral finite elements are employed. The error in the discrete
infinity norm is computed at $t=0.1$ by comparing it to a reference
solution with a sufficiently small step size. \label{fig:table-2d-1}}
\begin{centering}
\input{table-2d-1.tex}
\par\end{centering}
\end{table}

\end{example}

\begin{example} [Two-dimensional problem with inhomogeneous boundary condition]
Let us consider the initial value
\begin{equation}\label{eq:ivp-long}
\begin{aligned}
u_{0}(x,y)=0.5+2.0\Bigl(\mathrm{e}^{-40(x-0.5-0.1\cos\pi y)^{2}}&+\mathrm{e}^{-35(y-0.5-0.1\sin 2\pi x)^2}\\
&-\mathrm{e}^{-35((x-0.5)^{2}+(y-0.5)^{2})}\Bigr)
\end{aligned}
\end{equation}
and the corresponding compatible time independent boundary condition.
The initial value is shown in Figure~\ref{fig:initial-value} (left) and
the reference solution at time $t=0.1$ in Figure~\ref{fig:initial-value} (right).

\begin{figure}
\begin{centering}
\includegraphics[width=6.4cm]{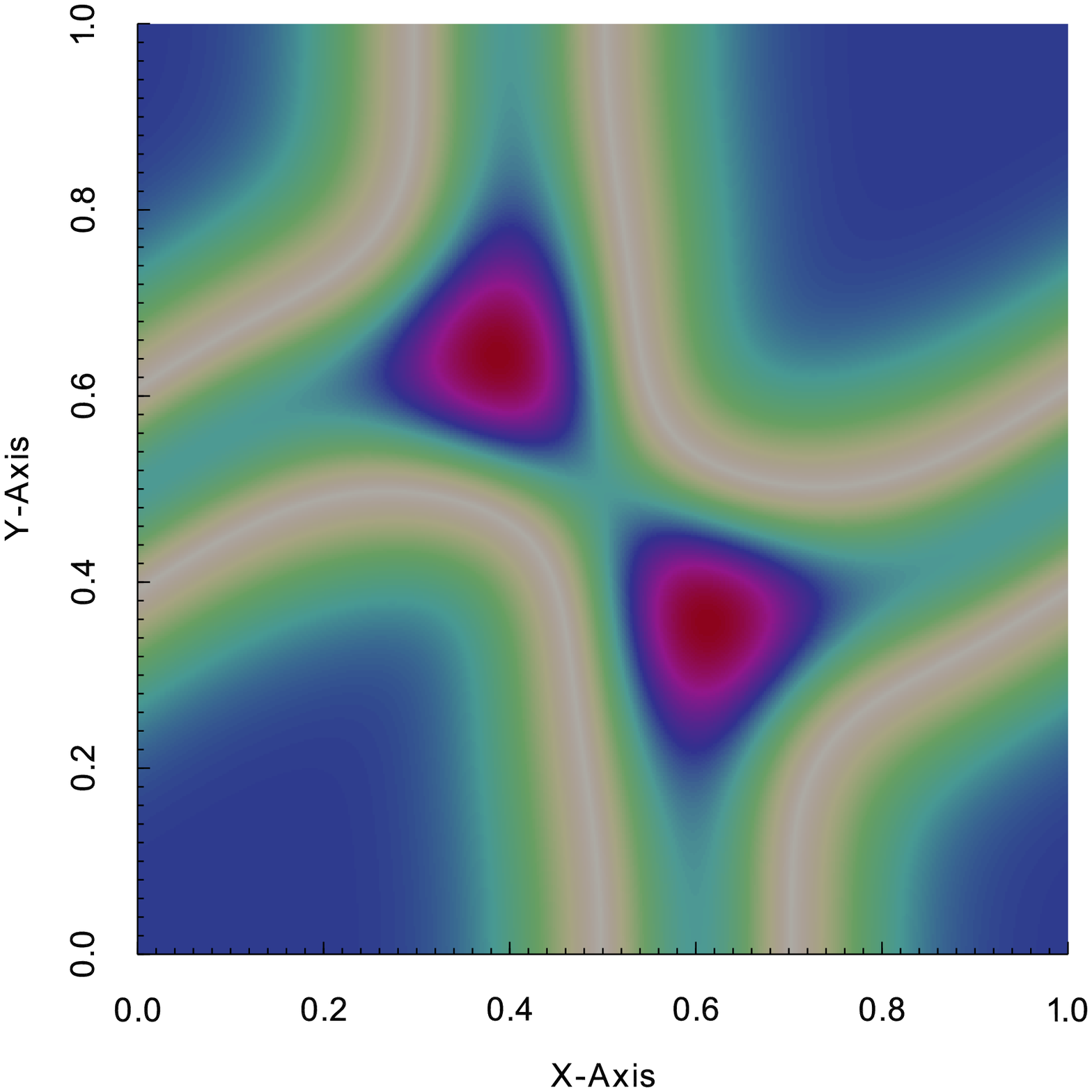}
\includegraphics[width=6.4cm]{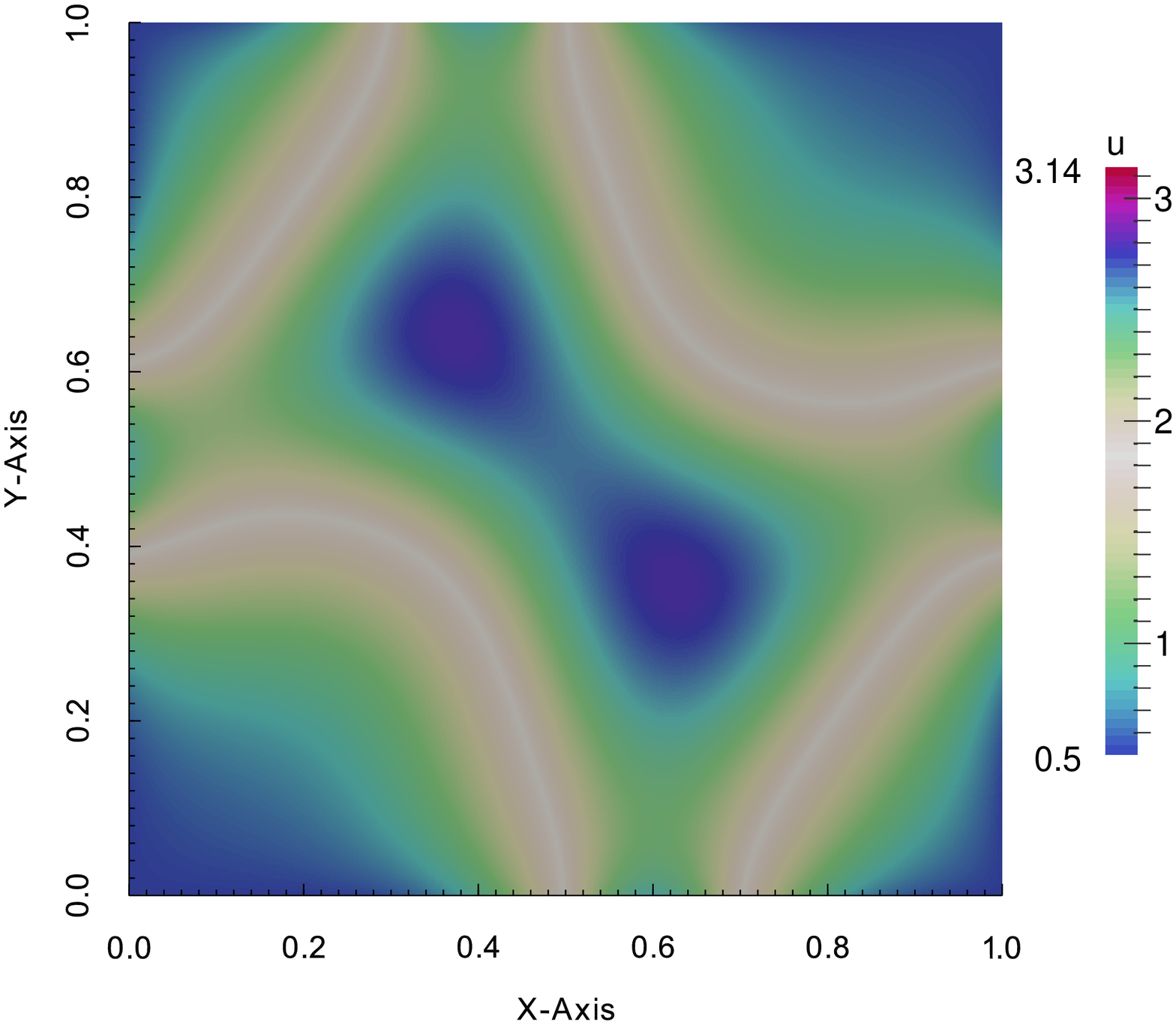}
\par\end{centering}
\vskip-3mm
\protect\caption{The initial value (left) and the reference solution at time $t=0.1$ (right) of equation~\eqref{eq:di-rea}, \eqref{eq:data2d}
with initial condition \eqref{eq:ivp-long}. The visualization was carried out with ParaView.
\label{fig:initial-value}}
\end{figure}

The results are shown in Table~\ref{fig:table-2d-inhom} and confirm
the order reduction in case of the classical Strang splitting as well
as that the modified scheme is in fact a second-order method.

\begin{table}
\protect\caption{Diffusion-reaction equation with the initial and boundary condition given in
equation (\ref{eq:ivp-long}). For the space discretization $2.5\cdot10^{5}$
quadrilateral finite elements are employed. The error in the discrete
infinity norm is computed at $t=0.1$ by comparing the numerical solution to a reference
solution with a sufficiently small step size. \label{fig:table-2d-inhom}}
\begin{centering}
\input{table-2d-inhom.tex}
\par\end{centering}
\end{table}

\end{example}

\section{Conclusion \& Outlook}

We have presented a splitting procedure that modifies both the diffusion
as well as the reaction partial flow in order to satisfy a compatibility
condition between the boundary conditions and the reaction term. This
yields a modified Strang splitting scheme that is of order two (i.e.,
no order reduction is observed) for inhomogeneous and even time dependent Dirichlet
boundary conditions. Crucially, the modification is independent of
the numerical solution and thus still allows us to take advantage
of the attractive features the splitting approach provides. In addition,
it has been observed that this modification for Lie splitting results
in better accuracy as compared to the classical Lie splitting scheme
(in certain problems up to a factor of $3.5$). Moreover, let us note
that the scheme is trivially generalizable to systems of diffusion-reaction
equations. The convergence analysis conducted shows that no
order reduction occurs for the modified Strang splitting. Furthermore, we show
that due to the parabolic smoothing property the classical Lie and Strang splitting
schemes are convergent of order one even though they are only consistent of order zero.

In a number of practical applications (such as those stemming from
combustion problems) Neumann boundary conditions are of interest. However,
this requires a more invasive modification of the splitting approach.
We consider this as future work.

\bibliographystyle{plain}
\bibliography{papers}

\end{document}

%% file: table-consistency.tex
\begin{tabular}{rrrrrr}
& \multicolumn{2}{c}{Lie} && \multicolumn{2}{c}{Lie (modified)}  \\
\cline{2-3} \cline{5-6}
\multicolumn{1}{l}{step size} & \multicolumn{1}{l}{local $l^{\infty}$ error} & \multicolumn{1}{l}{local order} & &
\multicolumn{1}{l}{local $l^{\infty}$ error} & \multicolumn{1}{l}{local order} \\
\hline
6.250e-03   &  1.250e-03   &  --     & &  7.682e-05   &  --  		\\
3.125e-03   &  6.237e-04   &  1.0033   & &  1.927e-05   &  1.9952   \\
1.563e-03   &  3.101e-04   &  1.0082   & &  4.826e-06   &  1.9974   \\
7.813e-04   &  1.551e-04   &  0.99946  & &  1.208e-06   &  1.9983 	\\
\end{tabular}

\vspace{2mm}
\begin{tabular}{rrrrrr}
& \multicolumn{2}{c}{Strang} && \multicolumn{2}{c}{Strang (modified)}  \\
\cline{2-3} \cline{5-6}
\multicolumn{1}{l}{step size} & \multicolumn{1}{l}{local $l^{\infty}$ error} & \multicolumn{1}{l}{local order} & &
\multicolumn{1}{l}{local $l^{\infty}$ error} & \multicolumn{1}{l}{local order} \\
\hline
6.250e-03   &  2.723e-03   &  --    &  &  3.716e-06   &  --   \\
3.125e-03   &  1.283e-03   &  1.0858  &  &  9.393e-07   &  1.9841 \\
1.563e-03   &  5.891e-04   &  1.123   &  &  2.352e-07   &  1.9978 \\
7.813e-04   &  2.602e-04   &  1.1789  &  &  5.866e-08   &  2.0034 \\
\end{tabular}

%% file: table-1d-1.tex
\begin{tabular}{rrrrrr}
& \multicolumn{2}{c}{Lie} && \multicolumn{2}{c}{Lie (modified)}  \\
\cline{2-3} \cline{5-6}
\multicolumn{1}{l}{step size} & \multicolumn{1}{l}{$l^{\infty}$ error} & \multicolumn{1}{l}{order} & &
\multicolumn{1}{l}{$l^{\infty}$ error} & \multicolumn{1}{l}{order} \\
\hline
2.000e-02   & 2.872e-01   & --      &    & 2.144e-01   & --          \\
1.000e-02   & 3.546e-03   & 6.3396  &    & 2.166e-03   & 6.6297      \\
5.000e-03   & 1.957e-03   & 0.85752 &    & 1.090e-03   & 0.99101     \\
2.500e-03   & 1.051e-03   & 0.89743 &    & 5.465e-04   & 0.99554     \\
1.250e-03   & 5.526e-04   & 0.92694 &    & 2.737e-04   & 0.99778     \\
6.250e-04   & 2.864e-04   & 0.94837 &    & 1.369e-04   & 0.99889     \\
3.125e-04   & 1.468e-04   & 0.96367 &    & 6.849e-05   & 0.99944     \\
\end{tabular}

\vspace{2mm}
\begin{tabular}{rrrrrr}
& \multicolumn{2}{c}{Strang} && \multicolumn{2}{c}{Strang (modified)}  \\
\cline{2-3} \cline{5-6}
\multicolumn{1}{l}{step size} & \multicolumn{1}{l}{$l^{\infty}$ error} & \multicolumn{1}{l}{order} & &
\multicolumn{1}{l}{$l^{\infty}$ error} & \multicolumn{1}{l}{order} \\
\hline
2.000e-02   & 9.371e-03   & --      &    & 3.013e-05   & --          \\
1.000e-02   & 4.519e-03   & 1.0521  &    & 7.540e-06   & 1.9985      \\
5.000e-03   & 2.156e-03   & 1.0678  &    & 1.885e-06   & 1.9999      \\
2.500e-03   & 1.010e-03   & 1.0939  &    & 4.709e-07   & 2.0011      \\
1.250e-03   & 4.603e-04   & 1.1337  &    & 1.173e-07   & 2.0047      \\
6.250e-04   & 2.013e-04   & 1.1931  &    & 2.896e-08   & 2.0185      \\
3.125e-04   & 8.281e-05   & 1.2817  &    & 6.923e-09   & 2.0647      \\
\end{tabular}

%% file: table-1d-1-l1l2.tex
\begin{tabular}{rrrrrr}
& \multicolumn{2}{c}{Strang} && \multicolumn{2}{c}{Strang(modified)}  \\
\cline{2-3} \cline{5-6}
\multicolumn{1}{l}{step size} & \multicolumn{1}{l}{$l^{1}$ error} & \multicolumn{1}{l}{order} & &
\multicolumn{1}{l}{$l^{1}$ error} & \multicolumn{1}{l}{order} \\
\hline
2.000e-02   & 4.679e-04   & --          & & 9.452e-06   & --          \\
1.000e-02   & 1.608e-04   & 1.5409      & & 2.362e-06   & 2.0008      \\
5.000e-03   & 5.511e-05   & 1.5449      & & 5.937e-07   & 1.992       \\
2.500e-03   & 1.884e-05   & 1.5487      & & 1.490e-07   & 1.9946      \\
1.250e-03   & 6.407e-06   & 1.556       & & 3.711e-08   & 2.0052      \\
\end{tabular}

\vspace{2mm}
\begin{tabular}{rrrrrr}
& \multicolumn{2}{c}{Strang} && \multicolumn{2}{c}{Strang (modified)}  \\
\cline{2-3} \cline{5-6}
\multicolumn{1}{l}{step size} & \multicolumn{1}{l}{$l^{2}$ error} & \multicolumn{1}{l}{order} & &
\multicolumn{1}{l}{$l^{2}$ error} & \multicolumn{1}{l}{order} \\
\hline
2.000e-02   & 1.524e-03   & --          & & 1.320e-05   & --          \\
1.000e-02   & 6.337e-04   & 1.2659      & & 3.303e-06   & 1.999       \\
5.000e-03   & 2.628e-04   & 1.2697      & & 8.264e-07   & 1.9987      \\
2.500e-03   & 1.085e-04   & 1.2766      & & 2.066e-07   & 1.9998      \\
1.250e-03   & 4.444e-05   & 1.2875      & & 5.152e-08   & 2.0039      \\
\end{tabular}

%% file: table-1d-timedep1.tex
\begin{tabular}{rrrrrr}
& \multicolumn{2}{c}{Lie} && \multicolumn{2}{c}{Lie (modified)}  \\
\cline{2-3} \cline{5-6}
\multicolumn{1}{l}{step size} & \multicolumn{1}{l}{$l^{\infty}$ error} & \multicolumn{1}{l}{order} & &
\multicolumn{1}{l}{$l^{\infty}$ error} & \multicolumn{1}{l}{order} \\
\hline
2.000e-02   & 2.872e-01   & --          & & 2.086e-01   & --          \\
1.000e-02   & 7.207e-03   & 5.3164      & & 8.593e-03   & 4.6015      \\
5.000e-03   & 4.053e-03   & 0.83053     & & 4.266e-03   & 1.0104      \\
2.500e-03   & 2.204e-03   & 0.87894     & & 2.125e-03   & 1.0052      \\
1.250e-03   & 1.172e-03   & 0.91124     & & 1.061e-03   & 1.0026      \\
6.250e-04   & 6.117e-04   & 0.93786     & & 5.298e-04   & 1.0013      \\
3.125e-04   & 3.158e-04   & 0.95371     & & 2.648e-04   & 1.0007      \\
\end{tabular}

\vspace{2mm}
\begin{tabular}{rrrrrr}
& \multicolumn{2}{c}{Strang} && \multicolumn{2}{c}{Strang (modified)}  \\
\cline{2-3} \cline{5-6}
\multicolumn{1}{l}{step size} & \multicolumn{1}{l}{$l^{\infty}$ error} & \multicolumn{1}{l}{order} & &
\multicolumn{1}{l}{$l^{\infty}$ error} & \multicolumn{1}{l}{order} \\
\hline
2.000e-02   & 2.060e-02   & --          & & 4.399e-04   & --          \\
1.000e-02   & 9.913e-03   & 1.0554      & & 1.099e-04   & 2.0005      \\
5.000e-03   & 4.724e-03   & 1.0694      & & 2.748e-05   & 2.0002      \\
2.500e-03   & 2.212e-03   & 1.0947      & & 6.867e-06   & 2.0005      \\
1.250e-03   & 1.008e-03   & 1.1341      & & 1.714e-06   & 2.002       \\
6.250e-04   & 4.407e-04   & 1.1932      & & 4.263e-07   & 2.0079      \\
3.125e-04   & 1.813e-04   & 1.2817      & & 1.043e-07   & 2.0316      \\
\end{tabular}

%% file: table-1d-timedep2.tex
\begin{tabular}{rrrrrr}
& \multicolumn{2}{c}{Lie} && \multicolumn{2}{c}{Lie (modified)}  \\
\cline{2-3} \cline{5-6}
\multicolumn{1}{l}{step size} & \multicolumn{1}{l}{$l^{\infty}$ error} & \multicolumn{1}{l}{order} & &
\multicolumn{1}{l}{$l^{\infty}$ error} & \multicolumn{1}{l}{order} \\
\hline
2.000e-02    & 5.525e-02   & --         & & 1.036e-01   & --          \\
1.000e-02    & 1.728e-03   & 4.9985     & & 2.609e-03   & 5.3112      \\
5.000e-03    & 8.643e-04   & 0.99986    & & 8.282e-04   & 1.6556      \\
2.500e-03    & 5.626e-04   & 0.61942    & & 3.257e-04   & 1.3463      \\
1.250e-03    & 3.426e-04   & 0.71549    & & 2.041e-04   & 0.67443     \\
6.250e-04    & 1.991e-04   & 0.78287    & & 1.184e-04   & 0.78556     \\
\end{tabular}

\vspace{2mm}
\begin{tabular}{rrrrrr}
& \multicolumn{2}{c}{Strang} && \multicolumn{2}{c}{Strang (modified)}  \\
\cline{2-3} \cline{5-6}
\multicolumn{1}{l}{step size} & \multicolumn{1}{l}{$l^{\infty}$ error} & \multicolumn{1}{l}{order} & &
\multicolumn{1}{l}{$l^{\infty}$ error} & \multicolumn{1}{l}{order} \\
\hline
2.000e-02   & 9.068e-03   & --          & & 3.757e-03   & --          \\
1.000e-02   & 4.405e-03   & 1.0418      & & 9.591e-04   & 1.9699      \\
5.000e-03   & 2.111e-03   & 1.0609      & & 2.410e-04   & 1.9927      \\
2.500e-03   & 9.913e-04   & 1.0907      & & 6.031e-05   & 1.9986      \\
1.250e-03   & 4.509e-04   & 1.1365      & & 1.506e-05   & 2.0013      \\
6.250e-04   & 1.951e-04   & 1.2083      & & 4.067e-06   & 1.889       \\
\end{tabular}

%% file: table-2d-1.tex
\begin{tabular}{rrrrrr}
& \multicolumn{2}{c}{Lie} && \multicolumn{2}{c}{Lie (modified)}  \\
\cline{2-3} \cline{5-6}
\multicolumn{1}{l}{step size} & \multicolumn{1}{l}{$l^{\infty}$ error} & \multicolumn{1}{l}{order} & &
\multicolumn{1}{l}{$l^{\infty}$ error} & \multicolumn{1}{l}{order} \\
\hline
0.1		&	1.436039e-01 & --		&&	    3.026891e-02 & --       \\
0.05	&	2.520559e-02 & 2.51028  &&	    8.189677e-03 & 1.88596\\
0.025	&	1.227821e-02 & 1.03764	&&      3.441324e-03 & 1.25084\\
0.0125	&	5.424341e-03 & 1.17858	&&      1.579346e-03 & 1.12364\\
\end{tabular}

\vspace{2mm}
\begin{tabular}{rrrrrr}
& \multicolumn{2}{c}{Strang} && \multicolumn{2}{c}{Strang (modified)}  \\
\cline{2-3} \cline{5-6}
\multicolumn{1}{l}{step size} & \multicolumn{1}{l}{$l^{\infty}$ error} & \multicolumn{1}{l}{order} & &
\multicolumn{1}{l}{$l^{\infty}$ error} & \multicolumn{1}{l}{order} \\
\hline
 0.1	&	 1.632674e-01& --       && 1.401885e-03& --       \\
 0.05	&	 1.445799e-01& 0.17537&& 3.806507e-04& 1.88083\\
 0.025	&	 1.062179e-01& 0.44484&& 9.978164e-05& 1.93162\\
 0.0125	&	 5.282179e-02& 1.00782&& 2.473435e-05& 2.01226\\
\end{tabular}

%% file: table-2d-inhom.tex
%

\begin{tabular}{rrrrrr}
& \multicolumn{2}{c}{Strang} && \multicolumn{2}{c}{Strang (modified)}  \\
\cline{2-3} \cline{5-6}
\multicolumn{1}{l}{step size} & \multicolumn{1}{l}{$l^{\infty}$ error} & \multicolumn{1}{l}{order} & &
\multicolumn{1}{l}{$l^{\infty}$ error} & \multicolumn{1}{l}{order} \\
\hline
0.1    & 8.449277e-01& --      &&1.835188e-02& --\\  
0.05   & 6.570760e-01& 0.362768&&4.962590e-03& 1.88676\\
0.025  & 4.063934e-01& 0.693183&&1.263375e-03& 1.97381\\
0.0125 & 1.670386e-01& 1.2827  &&3.326822e-04& 1.92507\\
\end{tabular}